\documentclass[11pt]{article}%
\usepackage{amsfonts}
\usepackage{amsmath}
\usepackage{amssymb}
\usepackage{graphicx}%
\providecommand{\U}[1]{\protect\rule{.1in}{.1in}}
\everymath{\displaystyle}
\newtheorem{theorem}{Theorem}

\newtheorem{proposition}[theorem]{Proposition}

\newenvironment{proof}[1][Proof]{\noindent\textbf{#1.} }{\ \rule{0.5em}{0.5em}}
\begin{document}

\date{}
\title{Integrals of Legendre polynomials and approximations}
\author{Abdelhamid Rehouma\\Department of mathematics Faculty of exact sciences\\University Hama Lakhdar, Eloued Algeria.\\URL: https://sites.google.com/view/mathsrehoumablog/accueil\\E-mail : rehoumaths@gmail.com}
\maketitle

\begin{abstract}
Let $\left\{  Q_{r}\left(  x\right)  \right\}  $ be a system of integral
Legendre polynomials of degree exactly $n$.The functions $Q_{i}\left(
x\right)  $ and $Q_{j}\left(  x\right)  $ ($i\neq j$) are orthogonal with
respect to the weight function $w\left(  x\right)  =\dfrac{1}{1-x^{2}}$, where
$-1\leq x\leq1.$We derive some identities and relations and extremal problems
and minimization and Fourier development involving of integral Legendre polynomials.

\end{abstract}

\bigskip

\textbf{Keywords:}{\scriptsize \ Legendre \ polynomials, PIPCIR
polynomials,\ Three-term rrecursive relation, Summations; Finite linear
combinations, Differentiation. \vspace{4mm} \newline}

\noindent\textbf{2000. MSC:}{\footnotesize \ 42C05, 33C45.}

\section{Introduction}

We restricted our attention to a polynomial with the first and last roots at
$x=\pm1$ ,given by
\begin{equation}
Q_{n}\left(  x\right)  =\left(  x^{2}-1\right)  q_{n-2}\left(  x\right)
\text{ \ \ \ \ ,}n\geq2 \label{Defi1}%
\end{equation}

Let us call a polynomial whose inflection points coincide with their interior
roots in a shorter way : Pipcir .It will be shown that the zeros of these
polynomials are all real, distinct, and they lie in the interval $\left[
-1\text{ \ \ \ \ }1\right]  .$The requirement all inflection points to
coincide with all roots of $Q_{n}\left(  x\right)  $ except $\pm1$ yields:%
\[
Q_{n}^{\prime\prime}\left(  x\right)  =-n\left(  n-1\right)  q_{n-2}\left(
x\right)
\]

or%
\begin{equation}
\left(  1-x^{2}\right)  Q_{n}^{\prime\prime}\left(  x\right)  +n\left(
n-1\right)  Q_{n}\left(  x\right)  =0 \label{Diff2}%
\end{equation}

Let us differentiate the equation \eqref{Diff2}%
\begin{equation}
-2xQ_{n}^{\prime\prime}\left(  x\right)  +\left(  1-x^{2}\right)
Q_{n}^{\prime\prime\prime}\left(  x\right)  +n\left(  n-1\right)
Q_{n}^{\prime}\left(  x\right)  =0 \label{Diff3}%
\end{equation}

We have now well-known Legendre's differential equation whose bounded on
$\left[  -1\text{ \ \ \ \ }1\right]  $ solutions are known as Legendre
polynomials:$y_{n}=L_{n-1}\left(  x\right)  $ ,$n\geq1.$One can find
properties of these polynomials in \cite{Abram},\cite{Belinsk},\cite{Pijeir}%
,they are normalized so that $L_{n}\left(  1\right)  =1$ for all $n$ .If%
\begin{equation}
Q_{n}\left(  x\right)  =-%
%TCIMACRO{\dint \limits_{x}^{1}}%
%BeginExpansion
{\displaystyle\int\limits_{x}^{1}}
%EndExpansion
L_{n-1}\left(  t\right)  dt\text{ \ \ \ \ \ }-1\leq x\leq1 \label{Qqn}%
\end{equation}

then $Q_{n}^{\prime}\left(  x\right)  =L_{n-1}\left(  x\right)  $ and
$Q_{n}^{\prime\prime}\left(  x\right)  =L_{n-1}^{\prime}\left(  x\right)  $
.We see that polynomials $Q_{n}\left(  x\right)  $ defined by \eqref{Qqn}
satisfy the equation \eqref{Diff2}.Thus,%
\begin{equation}
Q_{n}\left(  1\right)  =Q_{n}\left(  -1\right)  =0\text{ \ \ \ \ ,}n\geq2
\label{Qqn1}%
\end{equation}

We establish and We derive some identities and relations and extremal problems
and Fourier development involving of integral Legendre polynomials,also we
construct a new orthogonal systems by rational transforms in variable of
normalized integral Legendre polynomials with respect to a new weight function .

The Legendre Polynomial,$L_{n}\left(  x\right)  $ saisfies,\cite{Abram}%
,\cite{Belinsk},\cite{Pijeir} :%
\begin{equation}
L_{n}\left(  x\right)  =\frac{1}{2^{n}n!}\left(  \left(  x^{2}-1\right)
^{n}\right)  ^{\left(  n\right)  } \label{DifLn}%
\end{equation}
reads%

\[
L_{n}\left(  x\right)  =\frac{2n\left(  2n-1\right)  ..\left(  n+1\right)
}{2^{n}n!}\left(  x^{n}-\frac{n\left(  n-1\right)  }{2\left(  2n-1\right)
}x^{n-2}+\frac{n\left(  n-1\right)  \left(  n-2\right)  \left(  n-3\right)
}{2\left(  2n-1\right)  \left(  2n-3\right)  }x^{n-4}..+L_{n}\left(  0\right)
\right)
\]

The Legendre polynomials, denoted by $L_{k}\left(  x\right)  $, are the
orthogonal polynomials with $\omega\left(  x\right)  =1.$The three-term
recurrence relation for the Legendre polynomials reads,\cite{Abram}%
,\cite{Belinsk},\cite{Pijeir}%
\[
L_{0}\left(  x\right)  =1,L_{1}\left(  x\right)  =x,
\]
and%
\[
\left(  n+1\right)  L_{n+1}\left(  x\right)  =\left(  2n+1\right)
xL_{n}\left(  x\right)  -nL_{n-1}\left(  x\right)  \text{ \ \ \ \ \ }%
n=1,2,3..
\]
They are normalized so that $L_{n}\left(  1\right)  =1$ for all $n.$
yields,\cite{Abram},\cite{Belinsk},\cite{Pijeir}:%
\[
\left(  \left(  1-x^{2}\right)  L_{n}^{\prime}\left(  x\right)  \right)
^{\prime}+n\left(  n+1\right)  L_{n}\left(  x\right)  =0
\]
and%
\begin{equation}%
%TCIMACRO{\dint \limits_{-1}^{1}}%
%BeginExpansion
{\displaystyle\int\limits_{-1}^{1}}
%EndExpansion
L_{n}\left(  x\right)  L_{m}\left(  x\right)  dx=\frac{2}{2n+1}\delta
_{n,m},\text{ \ \ \ \ }n,m=1,2,3.. \label{orthLn}%
\end{equation}
and%
\[%
%TCIMACRO{\dint \limits_{-1}^{1}}%
%BeginExpansion
{\displaystyle\int\limits_{-1}^{1}}
%EndExpansion
L_{n}^{2}\left(  x\right)  =\frac{2n-1}{2n+1}%
%TCIMACRO{\dint \limits_{-1}^{1}}%
%BeginExpansion
{\displaystyle\int\limits_{-1}^{1}}
%EndExpansion
L_{n-1}^{2}\left(  x\right)  \text{ \ \ \ \ }n=1,2,3..
\]
We also derive that,\cite{Abram},\cite{Belinsk},\cite{Pijeir}%
\[%
%TCIMACRO{\dint \limits_{-1}^{x}}%
%BeginExpansion
{\displaystyle\int\limits_{-1}^{x}}
%EndExpansion
L_{n}\left(  t\right)  dt=\frac{1}{2n+1}\left(  L_{n+1}\left(  x\right)
-L_{n-1}\left(  x\right)  \right)
\]

We derive from above a recursive relation for computing the derivatives of the
Legendre polynomials,\cite{Abram},\cite{Belinsk},\cite{Pijeir}:
\[
L_{n}\left(  x\right)  =\frac{1}{2n+1}\left(  L_{n+1}^{\prime}\left(
x\right)  -L_{n-1}^{\prime}\left(  x\right)  \right)
\]
We also derive that,\cite{Abram}%
\begin{equation}
L_{n}\left(  \pm1\right)  =\left(  \pm1\right)  ^{n} \label{Lnat1}%
\end{equation}%
\[
L_{n}^{\prime}\left(  \pm1\right)  =\frac{1}{2}\left(  \pm1\right)
^{n-1}n\left(  n+1\right)
\]%
\[
L_{n}^{\prime\prime}\left(  \pm1\right)  =\left(  \pm1\right)  ^{n}\left(
n-1\right)  n\left(  n+1\right)  \left(  n+2\right)  /8
\]
We also derive that%
\begin{equation}
L_{n}\left(  x\right)  =\frac{1}{2^{n}}\sum_{k=0}^{n}\left(  C_{k}^{n}\right)
^{2}\left(  x-1\right)  ^{n-k}\left(  x+1\right)  ^{k} \label{expp}%
\end{equation}
becomes%
\begin{equation}
L_{n}\left(  0\right)  =\frac{1}{2^{n}}\sum_{k=0}^{n}\left(  -1\right)
^{n-k}\left(  C_{k}^{n}\right)  ^{2} \label{Lnat0}%
\end{equation}
and%
\begin{equation}
L_{2n}\left(  0\right)  =\frac{1}{2^{2n}}\left(  -1\right)  ^{n}C_{n}%
^{2n}=\left(  -1\right)  ^{n}\frac{1}{n!}\left(  \frac{1}{2}\right)  _{n}
\label{L2nat0}%
\end{equation}
and%
\begin{equation}
L_{2n+1}^{\prime}\left(  0\right)  =2\left(  -1\right)  ^{n}\frac{1}%
{n!}\left(  \frac{1}{2}\right)  _{n+1} \label{DeriLnat0}%
\end{equation}
and%
\begin{equation}
L_{n}^{\prime}\left(  0\right)  =\frac{1}{2^{n}}\sum_{k=0}^{n}\left(
-1\right)  ^{n-k}\left(  -n+2k\right)  \left(  C_{k}^{n}\right)  ^{2}
\label{Lnderivat0}%
\end{equation}

In fact,\cite{Belinsk}%
\begin{equation}
Q_{n}\left(  0\right)  =\frac{\left(  -1\right)  ^{\frac{n-2}{2}}\left(
n-3\right)  !!}{n!!} \label{Qnatzero}%
\end{equation}
The Rodrigues formula for the Gegenbauer polynomials is well known as the
following ,\cite{Belinsk}%
\begin{equation}
Q_{n}\left(  x\right)  =\frac{x^{2}-1}{2^{n-1}n!\left(  n-1\right)  }\left[
\left(  x^{2}-1\right)  ^{n-1}\right]  ^{\left(  n\right)  } \label{Rodrigues}%
\end{equation}

Now we have two expressions for $Q_{n}(x)$; equating them, we obtain the
formula%
\begin{equation}
\left(  x^{2}-1\right)  \left[  \left(  x^{2}-1\right)  ^{n-1}\right]
^{\left(  n\right)  }=n\left(  n-1\right)  \left[  \left(  x^{2}-1\right)
^{n-1}\right]  ^{\left(  n-2\right)  } \label{Second}%
\end{equation}

As is well known,\cite{Abram},we note that%
\begin{equation}%
%TCIMACRO{\dint \limits_{-1}^{1}}%
%BeginExpansion
{\displaystyle\int\limits_{-1}^{1}}
%EndExpansion
uv^{\left(  n\right)  }=\left\{
%TCIMACRO{\dsum \limits_{k=1}^{n}}%
%BeginExpansion
{\displaystyle\sum\limits_{k=1}^{n}}
%EndExpansion
\left(  -1\right)  ^{k-1}u^{\left(  k-1\right)  }v^{\left(  n-k\right)
}\right\}  _{-1}^{1}+\left(  -1\right)  ^{n}%
%TCIMACRO{\dint \limits_{-1}^{1}}%
%BeginExpansion
{\displaystyle\int\limits_{-1}^{1}}
%EndExpansion
u^{\left(  n\right)  }v \label{GenerIntegParts}%
\end{equation}
one can find%
\[
Q_{n}\left(  \pm1\right)  =0
\]%
\begin{equation}
Q_{n}^{\prime}\left(  1\right)  )=1 \label{Qnderiv1}%
\end{equation}%
\[
Q_{n}^{\prime\prime}\left(  1\right)  )=\frac{1}{2}\left(  \pm1\right)
^{n-1}n\left(  n-1\right)
\]%
\[
-2Q_{n}^{\prime\prime}\left(  1\right)  +n\left(  n-1\right)  Q_{n}^{\prime
}\left(  1\right)  =0
\]

We derive from above a recursive relation for computing the derivatives of the
\textbf{Legendr}e polynomials,\cite{Abram}:
\[
L_{n}\left(  x\right)  =\frac{1}{2n+1}\left(  L_{n+1}^{\prime}\left(
x\right)  -L_{n-1}^{\prime}\left(  x\right)  \right)
\]
reduces to%
\begin{equation}
Q_{n}^{\prime}\left(  x\right)  =\frac{1}{2n-1}\left(  Q_{n+1}^{\prime\prime
}\left(  x\right)  -Q_{n-1}^{\prime\prime}\left(  x\right)  \right)
\label{Pipcirs2}%
\end{equation}
Because for all $n,$:%
\[
Q_{n}\left(  1\right)  =Q_{n}\left(  -1\right)  =0
\]
Integrating both sides \ of \eqref{Pipcirs2} yields%
\begin{equation}
\int Q_{n}\left(  x\right)  dx=\frac{1}{2n-1}\left(  Q_{n+1}\left(  x\right)
-Q_{n-1}\left(  x\right)  \right)  \label{Pipcirs3}%
\end{equation}

We concentrate to prove the following proposition

\begin{proposition}
The functions $Q_{n}\left(  x\right)  $ and $Q_{n}\left(  x\right)  $ ($n\neq
m$) are orthogonal with respect to the weight function $w\left(  x\right)
=\dfrac{1}{1-x^{2}}$.Then
\end{proposition}

\begin{equation}%
%TCIMACRO{\dint \limits_{-1}^{1}}%
%BeginExpansion
{\displaystyle\int\limits_{-1}^{1}}
%EndExpansion
\dfrac{Q_{n}\left(  x\right)  Q_{m}\left(  x\right)  }{1-x^{2}}%
dx=0,\ \ \ \ \ (n\neq m),\ \ \ n,m=1,2,3.. \label{OrthQn}%
\end{equation}
and%
\begin{equation}
\left\Vert Q_{n}\right\Vert ^{2}=%
%TCIMACRO{\dint \limits_{-1}^{1}}%
%BeginExpansion
{\displaystyle\int\limits_{-1}^{1}}
%EndExpansion
\dfrac{Q_{n}^{2}\left(  x\right)  }{1-x^{2}}dx=\frac{2}{n\left(  n-1\right)
\left(  2n-1\right)  },\ \ \ \ \ \ \ \ \ \ \ n=2,3,4..... \label{NormQn}%
\end{equation}

\begin{proof}
In fact,%
\begin{equation}
L_{n}\left(  x\right)  =\frac{1}{2^{n}n!}\left(  \left(  x^{2}-1\right)
^{n}\right)  ^{\left(  n\right)  } \label{ExpLn}%
\end{equation}
and%
\begin{equation}
\left(  1-x^{2}\right)  Q_{n}^{\prime\prime}\left(  x\right)  +n\left(
n-1\right)  Q_{n}\left(  x\right)  =0 \label{DiffEqQn}%
\end{equation}
we can write%
\begin{equation}
Q_{n}\left(  x\right)  =-%
%TCIMACRO{\dint \limits_{x}^{1}}%
%BeginExpansion
{\displaystyle\int\limits_{x}^{1}}
%EndExpansion
L_{n-1}\left(  t\right)  dt\text{ }=\frac{-1}{2^{n-1}\left(  n-1\right)
!}\left(  \left(  x^{2}-1\right)  ^{n-1}\right)  ^{\left(  n-2\right)  }
\label{ExpQnn}%
\end{equation}
we obtain the formula%
\begin{equation}
\left(  x^{2}-1\right)  \left(  \left(  x^{2}-1\right)  ^{n-1}\right)
^{\left(  n\right)  }=n\left(  n-1\right)  \left(  \left(  x^{2}-1\right)
^{n-1}\right)  ^{\left(  n-2\right)  } \label{TwoDiffEq}%
\end{equation}
then, for $k=0,1,2,3.....n$%
\[
\dfrac{Q_{n}\left(  x\right)  x^{k}}{1-x^{2}}=\frac{-1}{n\left(  n-1\right)
}Q_{n}^{\prime\prime}\left(  x\right)  x^{k}=\frac{-1}{2^{n-1}n!\left(
n-1\right)  }\left(  \left(  x^{2}-1\right)  ^{n-1}\right)  ^{\left(
n\right)  }x^{k}%
\]
then, for $k=0,1,2,....n-1$%
\[%
%TCIMACRO{\dint \limits_{-1}^{1}}%
%BeginExpansion
{\displaystyle\int\limits_{-1}^{1}}
%EndExpansion
\dfrac{Q_{n}\left(  x\right)  x^{k}}{1-x^{2}}dx=\frac{-1}{2^{n-1}n!\left(
n-1\right)  }%
%TCIMACRO{\dint \limits_{-1}^{1}}%
%BeginExpansion
{\displaystyle\int\limits_{-1}^{1}}
%EndExpansion
\left(  \left(  x^{2}-1\right)  ^{n-1}\right)  ^{\left(  n\right)  }x^{k}dx
\]%
\[
\frac{-1}{2^{n-1}n!\left(  n-1\right)  }\left[  \left(  \left(  x^{2}%
-1\right)  ^{n-1}\right)  ^{\left(  n-1\right)  }x^{k}\right]  _{x=-1}%
^{x=1}+\frac{k}{2^{n-1}n!\left(  n-1\right)  }%
%TCIMACRO{\dint \limits_{-1}^{1}}%
%BeginExpansion
{\displaystyle\int\limits_{-1}^{1}}
%EndExpansion
\left(  \left(  x^{2}-1\right)  ^{n-1}\right)  ^{\left(  n-1\right)  }%
x^{k-1}dx
\]%
\[
=-\frac{k\left(  k-1\right)  }{2^{n-1}n!\left(  n-1\right)  }%
%TCIMACRO{\dint \limits_{-1}^{1}}%
%BeginExpansion
{\displaystyle\int\limits_{-1}^{1}}
%EndExpansion
\left(  \left(  x^{2}-1\right)  ^{n-1}\right)  ^{\left(  n-2\right)  }%
x^{k-2}dx
\]%
\[
.............................................................
\]%
\[
\pm\frac{k!}{2^{n-1}n!\left(  n-1\right)  }%
%TCIMACRO{\dint \limits_{-1}^{1}}%
%BeginExpansion
{\displaystyle\int\limits_{-1}^{1}}
%EndExpansion
\left(  \left(  x^{2}-1\right)  ^{n-1}\right)  ^{\left(  n-k\right)  }%
dx=\pm\frac{k!}{2^{n-1}n!\left(  n-1\right)  }\left[  \left(  \left(
x^{2}-1\right)  ^{n-1}\right)  ^{\left(  n-k-1\right)  }\right]  _{x=-1}%
^{x=1}=0
\]

\end{proof}

\section{The n-th kernel for the integral Legendre polynomials}

The $n$-th $Q$-kernel is given by, \cite{Foupouan}, \cite{Szeg}
\begin{equation}
K_{n}\left(  x,y\right)  =%
%TCIMACRO{\dsum \limits_{k=0}^{n}}%
%BeginExpansion
{\displaystyle\sum\limits_{k=0}^{n}}
%EndExpansion
\frac{Q_{k}(x)Q_{k}(y)}{\left\Vert Q_{k}\right\Vert ^{2}}. \label{Kern}%
\end{equation}
satisfies the \textbf{Christoffel-Darboux} formula, \cite{Foupouan},
\cite{Abramo},\cite{Szeg}%
\begin{equation}
K_{n}\left(  x,y\right)  =\frac{1}{\left\Vert Q_{n}\right\Vert ^{2}}%
\frac{Q_{n+1}(x)Q_{n}(y)-Q_{n+1}(y)Q_{n}(x)}{x-y},\ \ \ \ \ \ x\neq y
\label{CDS11}%
\end{equation}
and for $x=y$ one has
\begin{equation}
K_{n}\left(  x,x\right)  =\frac{1}{\left\Vert Q_{n}\right\Vert ^{2}}\left(
Q_{n+1}^{\prime}(x)Q_{n}(x)-Q_{n+1}(x)Q_{n}^{\prime}(x)\right)  . \label{ABC}%
\end{equation}
$K_{n}$ has the reproducing kernel property \cite{Foupouan}, \cite{Abramo}%
,\cite{Szeg}:%
\begin{equation}
f\left(  x\right)  =\int\limits_{-1}^{1}K_{n}\left(  x,t\right)  f\left(
t\right)  \frac{dt}{1-t^{2}} \label{Reprkernel}%
\end{equation}

\subsection{Extremal problem for integral Legendre polynomials}

Let $x\rightarrow w(x)$ be a nonnegative function on the interval $[-1,1]$
such that
\[
\int\limits_{-1}^{1}x^{r}w\left(  x\right)  dx
\]
exists for $r\geq0,$and consider the definite integral of the form%
\begin{equation}
I_{n}=\int\limits_{-1}^{1}\left(  b_{0}+b_{1}x+......b_{n}x^{n}\right)
^{2}\frac{dx}{1-x^{2}} \label{Integrl}%
\end{equation}

The problem to be solved is to determine the polynomial $x\longrightarrow
f_{n}\left(  x\right)  $ of order $n$ \ which minimizes the integral
\eqref{Integrl} under the constraint that $f_{n}\left(  0\right)  =1.$Since
the integrand is nonnegative for any value of $x\in\lbrack-1,1]$ such a
minimum value does exist.

Using standard minimization technique \cite{Foupouan}, \cite{Abramo}%
,\cite{Szeg} . It is well known that:%
\[
I_{n}=\int\limits_{-1}^{1}\left(  \sum_{k=0}^{n}a_{k}Q_{k}\left(  x\right)
\right)  ^{2}\frac{dx}{1-x^{2}}%
\]
and%
\begin{equation}
\sum_{k=0}^{n}a_{k}Q_{k}\left(  0\right)  =1 \label{Condition}%
\end{equation}
Denoting by%
\[
\left\Vert Q_{n}\right\Vert ^{2}=%
%TCIMACRO{\dint \limits_{-1}^{1}}%
%BeginExpansion
{\displaystyle\int\limits_{-1}^{1}}
%EndExpansion
\dfrac{Q_{n}^{2}\left(  x\right)  }{1-x^{2}}dx=\frac{2}{n\left(  n-1\right)
\left(  2n-1\right)  },\ \ \ \ \ \ \ \ \ \ \ n=2,3,4....
\]
we easily find ( in view of \eqref{Condition}) we can choosing \cite{Foupouan}%
, \cite{Abramo},\cite{Szeg} :
\[
a_{k}=\frac{Q_{k}\left(  0\right)  }{\left\Vert Q_{k}\right\Vert ^{2}}\frac
{1}{%
%TCIMACRO{\dsum \limits_{j=0}^{n}}%
%BeginExpansion
{\displaystyle\sum\limits_{j=0}^{n}}
%EndExpansion
\frac{Q_{j}\left(  0\right)  ^{2}}{\left\Vert Q_{j}\right\Vert ^{2}}}%
\]
so that the minimum value M of the integral \eqref{Integrl} under the
aforementioned constraint is%
\begin{equation}
M=\int\limits_{-1}^{1}\left(  \sum_{k=0}^{n}\frac{Q_{k}\left(  0\right)
M}{\left\Vert Q_{k}\right\Vert ^{2}}Q_{k}\left(  x\right)  \right)  ^{2}%
\frac{dx}{1-x^{2}}=\frac{1}{%
%TCIMACRO{\dsum \limits_{j=0}^{n}}%
%BeginExpansion
{\displaystyle\sum\limits_{j=0}^{n}}
%EndExpansion
\frac{Q_{j}\left(  0\right)  ^{2}}{\left\Vert Q_{j}\right\Vert ^{2}}}
\label{Mvalue}%
\end{equation}
and%
\begin{equation}
f_{n}\left(  x\right)  =\frac{1}{%
%TCIMACRO{\dsum \limits_{j=0}^{n}}%
%BeginExpansion
{\displaystyle\sum\limits_{j=0}^{n}}
%EndExpansion
\frac{Q_{j}\left(  0\right)  ^{2}}{\left\Vert Q_{j}\right\Vert ^{2}}}%
\sum_{k=0}^{n}\frac{Q_{k}\left(  0\right)  Q_{k}\left(  x\right)  }{\left\Vert
Q_{k}\right\Vert ^{2}} \label{Solution}%
\end{equation}
becomes%
\begin{equation}
f_{n}\left(  x\right)  =\sum_{k=0}^{n}\frac{MQ_{k}\left(  0\right)
Q_{k}\left(  x\right)  }{\left\Vert Q_{k}\right\Vert ^{2}} \label{Solution2}%
\end{equation}

\section{Main results.}

\begin{proposition}
The sequence $\left(  K_{n}\left(  x,0\right)  \right)  _{n=0}^{\infty}$ is
orthogonal with respect to the weight function
\[
t\left(  x\right)  =\frac{x}{1-x^{2}}%
\]
for \ $-1\leq x\leq1$, i. e.%
\[
\int\limits_{-1}^{1}K_{n}\left(  x,0\right)  K_{m}\left(  x,0\right)  \frac
{x}{1-x^{2}}dx=0,\ n\neq m.
\]
For all $n\geq1,$ one has%
\[
K_{n}\left(  0,0\right)  =
\]%
\begin{equation}
\frac{n\left(  n-1\right)  \left(  2n-1\right)  }{2}\left(  \frac{\left(
-1\right)  ^{\frac{n-2}{2}}\left(  n-3\right)  !!}{2^{n}n!!}\sum_{k=0}%
^{n}\left(  -1\right)  ^{n-k}\left(  C_{k}^{n}\right)  ^{2}-\frac{\left(
-1\right)  ^{\frac{n-1}{2}}\left(  n-2\right)  !!}{2^{n-1}\left(  n+1\right)
!!}\sum_{k=0}^{n-1}\left(  -1\right)  ^{n-k+1}\left(  C_{k}^{n-1}\right)
^{2}\right)  \label{Knn00}%
\end{equation}

\end{proposition}

\begin{proof}
by \eqref{NormQn},\eqref{Qnatzero}%
\begin{equation}
K_{n}\left(  x,0\right)  =%
%TCIMACRO{\dsum \limits_{k=0}^{n}}%
%BeginExpansion
{\displaystyle\sum\limits_{k=0}^{n}}
%EndExpansion
\frac{Q_{k}(x)Q_{k}(0)}{\left\Vert Q_{k}\right\Vert ^{2}}=%
%TCIMACRO{\dsum \limits_{k=0}^{n}}%
%BeginExpansion
{\displaystyle\sum\limits_{k=0}^{n}}
%EndExpansion
\frac{\left(  -1\right)  ^{\frac{k-2}{2}}k\left(  k-1\right)  \left(
2k-1\right)  \left(  k-3\right)  !!}{2k!!}Q_{k}(x) \label{Knxat0}%
\end{equation}
by \eqref{ABC},\eqref{Qnatzero}%
\[
K_{n}\left(  0,0\right)  =\frac{n\left(  n-1\right)  \left(  2n-1\right)  }%
{2}\left(  Q_{n+1}^{\prime}(0)\frac{\left(  -1\right)  ^{\frac{n-2}{2}}\left(
n-3\right)  !!}{n!!}-\frac{\left(  -1\right)  ^{\frac{n-1}{2}}\left(
n-2\right)  !!}{\left(  n+1\right)  !!}Q_{n}^{\prime}(0)\right)
\]
from which it follows%
\begin{equation}
K_{n}\left(  0,0\right)  =\frac{n\left(  n-1\right)  \left(  2n-1\right)  }%
{2}\left(  L_{n}(0)\frac{\left(  -1\right)  ^{\frac{n-2}{2}}\left(
n-3\right)  !!}{n!!}-\frac{\left(  -1\right)  ^{\frac{n-1}{2}}\left(
n-2\right)  !!}{\left(  n+1\right)  !!}L_{n-1}(0)\right)  \label{Kn0at0}%
\end{equation}
Now comming back to the\eqref{Lnat0},gives%
\[
L_{n}\left(  0\right)  =\frac{1}{2^{n}}\sum_{k=0}^{n}\left(  -1\right)
^{n-k}\left(  C_{k}^{n}\right)  ^{2}%
\]
Using \eqref{Kn0at0}, \eqref{Lnat0}, we have the desired result \eqref{Knn00} .
\end{proof}

\begin{theorem}
the integral
\[
I_{n}=\int\limits_{-1}^{1}\left(  F_{n}\left(  x\right)  \right)  ^{2}%
\frac{dx}{1-x^{2}}%
\]

where $F_{n}$ is any real polynomial of degree $n$ such that $F_{n}%
(0)=1,$reaches its minimum value
\begin{equation}
M=\frac{1}{%
%TCIMACRO{\dsum \limits_{j=0}^{n}}%
%BeginExpansion
{\displaystyle\sum\limits_{j=0}^{n}}
%EndExpansion
\frac{j\left(  j-1\right)  \left(  2j-1\right)  }{2}\left(  \frac{\left(
j-3\right)  !!}{j!!}\right)  ^{2}} \label{Valuem}%
\end{equation}

if and only if%
\begin{equation}
F_{n}\left(  x\right)  =\frac{1}{%
%TCIMACRO{\dsum \limits_{j=0}^{n}}%
%BeginExpansion
{\displaystyle\sum\limits_{j=0}^{n}}
%EndExpansion
\frac{j\left(  j-1\right)  \left(  2j-1\right)  }{2}\left(  \frac{\left(
j-3\right)  !!}{j!!}\right)  ^{2}}\sum_{k=0}^{n}\left(  -1\right)
^{\frac{k-2}{2}}\frac{k\left(  k-1\right)  \left(  2k-1\right)  \left(
k-3\right)  !!}{2k!!}Q_{k}\left(  x\right)  \label{Fnnx}%
\end{equation}

ie%
\begin{equation}
M=\frac{1}{K_{n}\left(  0,0\right)  } \label{Kernelm}%
\end{equation}

and%
\begin{equation}
F_{n}\left(  x\right)  =\frac{K_{n}\left(  x,0\right)  }{K_{n}\left(
0,0\right)  } \label{Kernelf}%
\end{equation}

\end{theorem}

\begin{proof}
Using \eqref{Integrl}, \eqref{Mvalue}, \eqref{Solution},
\eqref{Qnatzero},\eqref{NormQn},gives the minimum value%
\[
M=\frac{1}{%
%TCIMACRO{\dsum \limits_{j=0}^{n}}%
%BeginExpansion
{\displaystyle\sum\limits_{j=0}^{n}}
%EndExpansion
\frac{j\left(  j-1\right)  \left(  2j-1\right)  }{2}\left(  \frac{\left(
j-3\right)  !!}{j!!}\right)  ^{2}}%
\]
and
\[
F_{n}\left(  x\right)  =\frac{1}{%
%TCIMACRO{\dsum \limits_{j=0}^{n}}%
%BeginExpansion
{\displaystyle\sum\limits_{j=0}^{n}}
%EndExpansion
\frac{j\left(  j-1\right)  \left(  2j-1\right)  }{4}\left(  \frac{\left(
j-3\right)  !!}{j!!}\right)  ^{2}}\sum_{k=0}^{n}\left(  -1\right)
^{\frac{k-2}{2}}\frac{k\left(  k-1\right)  \left(  2k-1\right)  \left(
k-3\right)  !!}{2k!!}Q_{k}\left(  x\right)
\]
and this completes the proof of Theorem .
\end{proof}

\section{Fourier development and approximations}

\bigskip Now for a function $f\left(  x\right)  $ expandable in these
polynomials one gets,%
\[
f\left(  x\right)  =\sum_{n=0}^{\infty}a_{n}Q_{n}\left(  x\right)
\]
where%
\begin{equation}
a_{n}=\frac{n\left(  n-1\right)  \left(  2n-1\right)  }{2}%
%TCIMACRO{\dint \limits_{-1}^{1}}%
%BeginExpansion
{\displaystyle\int\limits_{-1}^{1}}
%EndExpansion
f\left(  x\right)  Q_{n}\left(  x\right)  \dfrac{dx}{1-x^{2}} \label{FourierQ}%
\end{equation}

\begin{proposition}
If%
\[
f\left(  x\right)  =\sum_{n=0}^{\infty}a_{n}Q_{n}\left(  x\right)
\]
then%
\begin{equation}
a_{n}=\frac{\left(  -1\right)  ^{n}\left(  2n-1\right)  }{2^{n}\left(
n-1\right)  !}\sum_{k=0}^{n}\left(  -1\right)  ^{k}C_{k}^{n-1}\mu_{2k}
\label{anex}%
\end{equation}
where%
\[
\mu_{2k}=\mu_{2k}\left(  f^{\left(  n\right)  }\right)  =%
%TCIMACRO{\dint \limits_{-1}^{1}}%
%BeginExpansion
{\displaystyle\int\limits_{-1}^{1}}
%EndExpansion
x^{2k}f^{\left(  n\right)  }\left(  x\right)  dx
\]

\end{proposition}

For $\ $the Fourier polynomial
\[
P_{n}\left(  x\right)  =\sum_{k=0}^{n}a_{k}Q_{n}\left(  x\right)
\]
If $f\left(  x\right)  =x^{k},$then%
\begin{equation}
a_{k}=\frac{\left(  -1\right)  ^{k}k2^{k-1}\left(  \left(  k-1\right)
!\right)  ^{2}}{\left(  2k-2\right)  !}\text{ \ \ \ \ , }k=1,2,3,4......n
\label{akk}%
\end{equation}

\begin{quote}

\end{quote}

\begin{proof}
Employing \eqref{Rodrigues},\eqref{FourierQ}, yields ,%
\[
a_{n}=\frac{\left(  2n-1\right)  }{2^{n}\left(  n-1\right)  !}%
%TCIMACRO{\dint \limits_{-1}^{1}}%
%BeginExpansion
{\displaystyle\int\limits_{-1}^{1}}
%EndExpansion
f\left(  x\right)  \left[  \left(  x^{2}-1\right)  ^{n-1}\right]  ^{\left(
n\right)  }dx
\]
Combining \eqref{anexpansion},\eqref{GenerIntegParts} we get%
\begin{equation}
=\frac{\left(  -1\right)  ^{n}\left(  2n-1\right)  }{2^{n}\left(  n-1\right)
!}%
%TCIMACRO{\dint \limits_{-1}^{1}}%
%BeginExpansion
{\displaystyle\int\limits_{-1}^{1}}
%EndExpansion
\left(  1-x^{2}\right)  ^{n-1}f^{\left(  n\right)  }\left(  x\right)  dx
\label{anexpansion}%
\end{equation}%
\[
=\frac{\left(  -1\right)  ^{n}\left(  2n-1\right)  }{2^{n}\left(  n-1\right)
!}\sum_{k=0}^{n}\left(  -1\right)  ^{k}C_{k}^{n-1}%
%TCIMACRO{\dint \limits_{-1}^{1}}%
%BeginExpansion
{\displaystyle\int\limits_{-1}^{1}}
%EndExpansion
x^{2k}f^{\left(  n\right)  }\left(  x\right)  dx
\]
So the proof of the \eqref{anex} is completed.Therefore, by
\eqref{anexpansion}, \
\[
a_{k}=\frac{\left(  -1\right)  ^{k}\left(  2k-1\right)  }{2^{k}\left(
k-1\right)  !}%
%TCIMACRO{\dint \limits_{-1}^{1}}%
%BeginExpansion
{\displaystyle\int\limits_{-1}^{1}}
%EndExpansion
\left(  1-x^{2}\right)  ^{k-1}f^{\left(  k\right)  }\left(  x\right)  dx
\]%
\[
=\frac{\left(  -1\right)  ^{k}\left(  2k-1\right)  k}{2^{k}}%
%TCIMACRO{\dint \limits_{-1}^{1}}%
%BeginExpansion
{\displaystyle\int\limits_{-1}^{1}}
%EndExpansion
\left(  1-x^{2}\right)  ^{k-1}dx
\]%
\[
=\frac{\left(  -1\right)  ^{k}\left(  2k-1\right)  k}{2^{k-1}}%
%TCIMACRO{\dint \limits_{0}^{1}}%
%BeginExpansion
{\displaystyle\int\limits_{0}^{1}}
%EndExpansion
\left(  1-x^{2}\right)  ^{k-1}dx
\]%
\[
=\frac{\left(  -1\right)  ^{k}\left(  2k-1\right)  k}{2^{k-1}}%
%TCIMACRO{\dint \limits_{0}^{\frac{\pi}{2}}}%
%BeginExpansion
{\displaystyle\int\limits_{0}^{\frac{\pi}{2}}}
%EndExpansion
\sin^{2k-1}xdx=\frac{\left(  -1\right)  ^{k}\left(  2k-1\right)  k}{2^{k-1}%
}\frac{4^{k-1}\left(  \left(  k-1\right)  !\right)  ^{2}}{\left(  2k-1\right)
!}%
\]%
\[
=\frac{\left(  -1\right)  ^{k}k2^{k-1}\left(  \left(  k-1\right)  !\right)
^{2}}{\left(  2k-2\right)  !}%
\]
and this completes the proof.of \eqref{anex} and \eqref{akk}.
\end{proof}

\begin{theorem}
Let $f$ be an increasing function on $\left[  -1\text{ \ \ \ }1\right]  $
,with $f\left(  a\right)  =-1$ and $f\left(  b\right)  =1,$such that $a<b$ and
$\varphi$ a nonnegative weight function on the same interval, such that the
integral
\[
\int\limits_{-1}^{1}f\left(  x\right)  ^{n}\varphi\left(  x\right)
dx\ \ \ \ \ \ \ \ \ \ \ \left(  n\geq0\right)
\]

exists; Then the sequence of functions $x\mapsto r_{0}\left(  f\left(
x\right)  \right)  ,x\mapsto r_{1}\left(  f\left(  x\right)  \right)
,...x\mapsto r_{n}\left(  f\left(  x\right)  \right)  ...$that minimizes the
integrals%
\begin{equation}
I_{n}=\int\limits_{a}^{b}q_{n}\left(  f\left(  x\right)  \right)  ^{2}%
\varphi\left(  x\right)  dx \label{In}%
\end{equation}

for all polynomial : $q_{n}\left(  x\right)  =b_{0}+b_{1}x+......b_{n}x^{n}$,
forms an orthogonal system on $\left[  a\text{ \ \ \ }b\right]  $ in respect
of \ $\varphi.$Where%
\begin{equation}
\frac{\varphi\left(  x\right)  }{f^{\prime}\left(  x\right)  }=\left(
1-f^{2}\left(  x\right)  \right)  \label{weight}%
\end{equation}

i.e,%
\[
\int\limits_{a}^{b}r_{n}\left(  f\left(  x\right)  \right)  r_{m}\left(
f\left(  x\right)  \right)  \left(  1-f^{2}\left(  x\right)  \right)
f^{\prime}\left(  x\right)  dx=0\text{ \ \ \ \ \ \ \ \ ,}n=0,1,2.....\left(
n\neq m\right)
\]

If
\begin{equation}
f\left(  x\right)  =\frac{\lambda x+\alpha}{\mu x+\beta}\text{ , \ such
that\ \ \ \ }\lambda\beta-\mu\alpha=1 \label{fff}%
\end{equation}

satisfie, $f\left(  \frac{\beta-\alpha}{\lambda+\mu}\right)  =-1,f\left(
\frac{\beta-\alpha}{\lambda-\mu}\right)  =1,$then%
\[
\int\limits_{\frac{\beta-\alpha}{\lambda+\mu}}^{\frac{\beta-\alpha}%
{\lambda-\mu}}r_{n}\left(  \frac{\lambda x+\alpha}{\mu x+\beta}\right)
r_{m}\left(  \frac{\lambda x+\alpha}{\mu x+\beta}\right)  \varphi\left(
x\right)  dx=0\ \ \ \ \ \ \ \ ,n=0,1,2.....\left(  n\neq m\right)
\]

where%
\begin{equation}
\varphi\left(  x\right)  =\frac{\left(  \left(  \mu-\lambda\right)
x+\beta-\alpha\right)  \left(  \left(  \mu+\lambda\right)  x+\beta
+\alpha\right)  }{\left(  \mu x+\beta\right)  ^{4}} \label{wffff}%
\end{equation}

Then the sequence of functions%
\[
x\mapsto r_{0}\left(  \frac{\lambda x+\alpha}{\mu x+\beta}\right)  ,x\mapsto
r_{1}\left(  \frac{\lambda x+\alpha}{\mu x+\beta}\right)  ,...x\mapsto
r_{n}\left(  \frac{\lambda x+\alpha}{\mu x+\beta}\right)
\]

forms an orthogonal system on $\left[  \frac{\beta-\alpha}{\lambda+\mu}\text{
\ \ \ }\frac{\beta-\alpha}{\lambda-\mu}\right]  $ in respect of \ $\varphi.$
\end{theorem}

\begin{proof}
the polynomials $r_{n};$ $(n=0,1,2,...)$are orthogonal on $[-1$ $\ \ \ 1]$ in
respect of $t\mapsto1-t^{2}$.i-e%
\[
\int\limits_{-1}^{1}r_{n}\left(  t\right)  r_{m}\left(  t\right)  \left(
1-t^{2}\right)  dt=0\text{ \ \ \ \ \ \ \ \ ,}n=0,1,2.....\left(  n\neq
m\right)
\]

Substituting $f(x)=t$ \ in \eqref{In} we have%
\[
I_{n}=\int\limits_{a}^{b}q_{n}\left(  t\right)  ^{2}\frac{\varphi\left(
f^{-1}\left(  t\right)  \right)  }{f^{\prime}\left(  f^{-1}\left(  t\right)
\right)  }dt\ \ \ \ \ \ ,n=0,1,2.....
\]

if%
\[
\frac{\varphi\left(  f^{-1}\left(  t\right)  \right)  }{f^{\prime}\left(
f^{-1}\left(  t\right)  \right)  }=1-t^{2}%
\]

Now comming back to the old variable with according to Theorem , the
minimizing functions
\[
x\mapsto r_{0}\left(  f\left(  x\right)  \right)  ,x\mapsto r_{1}\left(
f\left(  x\right)  \right)  ,x\mapsto r_{2}\left(  f\left(  x\right)  \right)
,...x\mapsto r_{n}\left(  f\left(  x\right)  \right)  ....
\]

that minimize \eqref{In} form an orthogonal system on $\left[  a\text{
\ \ \ }b\right]  $ in respect of \ $\varphi.$Therefore we denote as%
\[
x\longrightarrow r_{k}\left(  \frac{\lambda x+\alpha}{\mu x+\beta}\right)
,\text{ \ \ }k=2,3,4........
\]

form an orthogonal system on $\left[  \frac{\beta-\alpha}{\lambda+\mu}\text{
\ \ \ }\frac{\beta-\alpha}{\lambda-\mu}\right]  $ in respect of \ $\varphi$%
\[
\varphi\left(  x\right)  =\frac{\left(  \left(  \mu-\lambda\right)
x+\beta-\alpha\right)  \left(  \left(  \mu+\lambda\right)  x+\beta
+\alpha\right)  }{\left(  \mu x+\beta\right)  ^{4}}%
\]

and this completes the proof of Theorem
\end{proof}

\end{document}